\documentclass[11pt,reqno]{amsart}

\usepackage{amsmath,amssymb,mathrsfs,amsfonts,mathabx}
\usepackage{graphicx,cite,cases}
\usepackage{xcolor}
\usepackage{epstopdf}

\newtheorem{theorem}{Theorem}[section]

\newtheorem{remark}[theorem]{Remark}

\newtheorem{assumption}{Assumption}
\newtheorem{definition}{Definition}

\numberwithin{equation}{section}

\begin{document}

\title[an inverse source problem]{An inverse source problem for the stochastic wave equation}

\author{Xiaoli Feng}
\address{School of Mathematics and Statistics, Xidian University, Xi'an, 713200, P. R. China}
\email{xiaolifeng@xidian.edu.cn}

\author{Meixia Zhao}
\address{School of Mathematics and Statistics, Xidian University, Xi'an, 713200, P. R. China}
\email{meixiazhao@stu.xidian.edu.cn}

\author{Peijun Li}
\address{Department of Mathematics, Purdue University, West Lafayette, Indiana 47907, USA}
\email{lipeijun@math.purdue.edu}

\author{Xu Wang}
\address{Department of Mathematics, Purdue University, West Lafayette, Indiana 47907, USA}
\email{wang4191@purdue.edu}

\thanks{The research of XF is supported partially by the CSC fund (No. 201806965033). The research of PL is supported in part by the NSF grant DMS-1912704.}



\begin{abstract}
This paper is concerned with an inverse source problem for the stochastic wave equation driven by a fractional Brownian motion. Given the random source, the direct problem is to study the solution of the stochastic wave equation. The inverse problem is to determine the statistical properties of the source from the expectation and covariance of the final-time data. For the direct problem, it is shown to be well-posed with a unique mild solution. For the inverse problem, the uniqueness is proved for a certain class of functions and the instability is characterized. Numerical experiments are presented to illustrate the reconstructions by using a truncation-based regularization method.
\end{abstract}

\maketitle

\section{Introduction}

As one of the representative examples on inverse problems in partial differential equations, the inverse source problem (ISP) has received a lot of attention in mathematical and engineering communities. Especially, the ISP for wave propagation is an active and important research topic in the field of inverse problems due to the significant applications in diverse scientific areas, such as magnetoencephalography \cite{magnetoencephalography}, photoacoustic tomography \cite{photoacoustic}, ultrasonics imaging \cite{ultrasonics}, antenna design and synthesis \cite{Antennas1,Antennas2}.

The ISP is to determine the unknown source from a knowledge about the solution. It is challenging due to the ill-posedness for lack of uniqueness or stability. This paper is concerned with an ISP for the wave equation, which has been 
extensively investigated for the deterministic case. The well-posedness and stability can be found in \cite{BLT-jde10, Cannon+1983, Uniqueness+1981, Liu+Triggiani+2011, Maarten+2016, Uniqueness+Stability+Yamamoto+2003, Uniqueness+Stability+Yamamoto, Uniqueness+Stability+2011} and \cite{Liu+Triggiani+2011, Uniqueness+Stability+Yamamoto+2003, Uniqueness+Stability+Yamamoto, Stability+Imanuvilov+2001, Uniqueness+Stability+2011, Stability+Yamamoto+1995}, respectively. Some of the numerical results may be found in \cite{BaoG+2011, Chen+2020, Hasanov+2016, Nguyen+2019, Sattari+2019} and the references cited therein. We refer to the monograph \cite{theory+Isakov+1990} for a complete account of the general theory on the ISP for the wave equation.

Recently, the field of stochastic inverse problems has been undergoing a rapid development and progressing to an area of intense activity. Stochastic inverse problems refer to inverse problems that involve uncertainties, which play a vital role in mathematical models to handle unpredictability of the environments and incomplete knowledge of the systems and measurements. Compared to deterministic counterparts, stochastic inverse problems have substantially more difficulties 
on top of existing hurdles due to randomness and uncertainties. The inverse random source problems for the time-harmonic wave equations have been widely studied. In \cite{Devaney+1979}, it was shown that the correlation of the random source could be determined uniquely by the correlation of the random wave field. Effective computational models were developed in \cite{Bao+Chen+Li+2016, Bao+Chen+Li+2017, Bao+Chow+Li+Zhou+2014, Li+Chen+Li+2017, LiPJ+2011, Li+Yuan+2017, LW2021} for the time-harmonic stochastic acoustic and elastic wave equations, where the goal was to reconstruct the statistical properties of the random source from the boundary measurement  of the radiated random wave field at multiple frequencies.

In this paper, we consider an ISP for the following initial-boundary value problem of the stochastic wave equation driven by the fractional Brownian motion (fBm):
\begin{equation}\label{directp}
\begin{cases}
u_{tt}(x,t)-\Delta u(x,t)=F(x,t), &\quad (x,t)\in{D}\times(0,T),\\
u(x,t)=0, & \quad (x,t)\in{\partial D}\times[0,T], \\
u(x,0)=u_t(x,0)=0, &\quad x\in\overline{D},
\end{cases}
\end{equation}
where $D\subset\mathbb R^d$ is a bounded domain with Lipschitz boundary $\partial D$ and the random source is assumed to take the form
\begin{equation}\label{rs}
F(x,t)=f(x)h(t)+g(x)\dot{B}^H(t).
\end{equation}
Here $f(x)$ and $g(x)$ are deterministic functions with compact supports contained in $D$, $h(t)$ is also a deterministic function, $B^H(t)$ is the fBm with the Hurst index $H\in(0,1)$, and $\dot{B}^H(t)$ can be roughly understood as the derivative of $B^H(t)$ with respect to the time $t$. When $H=\frac12$, the fBm reduces to the classical Brownian motion and $\dot{B}^H(t)$ becomes the white noise. Since the source $F(x,t)$ is a random field with low regularity, it is a distribution instead of a function. More precisely, it is shown in \cite{LW2019} that $F(\cdot,t)\in W^{H-1-\epsilon}(\mathbb{R})$ for any $\epsilon>0$. Clearly, we can see that $H-1-\epsilon<0$ as $H\in(0,1)$, and the smaller the Hurst index $H$ is, the lower the regularity of the source has. Given $F$, the direct problem is to determine the wave field $u$; the inverse problem is to recover $f$ and $g$ from the final-time data, i.e., $u(x, T), x\in D$.

So far, there have been considerable work done for the direct problems of the stochastic wave equation with different random sources. In \cite{Orsingher+1982}, the author studied randomly forced vibrations of a string driven by three Gaussian white noises. In \cite{Walsh+1984}, a white noise in both time and space was considered, and the existence and uniqueness of the solution were obtained for the stochastic wave equation. In \cite{Dalang+1998}, the authors examined the conditions on the well-posedness for a two-dimensional stochastic wave equation driven by a space-time Gaussian noise that is white in time but has a non-degenerate spatial covariance. In \cite{Caithamer+2005}, the author gave the upper and lower bounds on both the large and small derivations of several sup norms associated with the solution for a fractional Brownian noise. Some existence and uniqueness results can be found in \cite{Tindel+2007} for the one-dimensional stochastic wave equation driven by a two-parameter fBm. In \cite{
Erraoui+2008}, the existence and uniqueness of solutions were investigated for a class of hyperbolic stochastic partial differential equations driven by a space-time additive fractional Brownian sheet. In \cite{Tang+2020}, the authors showed several solutions of a stochastic wave
equation in the plane with an additive noise which is fractional in time and has a non-degenerate spatial covariance. In
\cite{numerical+1d+2017}, the existence of solutions was obtained and Newton's method was applied for nonlinear stochastic wave equations driven by one-dimensional white noise with respect to time. In \cite{Dengnumerical+2020},
a higher order approximation was proposed to solve the stochastic space fractional wave equation forced by an additive space-time Gaussian noise.

Compared with the direct problems for stochastic wave equations driven by fBm, there are few work for the inverse source problems for the stochastic wave equations driven by fBm. Recently, \cite{LW2019} considered an inverse random source problem for the Helmholtz equation driven by a fractional Gaussian field. The approach was further extended to solve the inverse random source problem for time-harmonic Maxwell's equations driven by a centered complex-valued Gaussian vector field with correlated components \cite{LW2020}. In \cite{Niu+2018} and \cite{FLW2020}, the ISP was studied for the stochastic fractional diffusion equation where the source is assumed to take the form of \eqref{rs} for $H=\frac{1}{2}$ and a general $H\in(0, 1)$, respectively. The goal is to reconstruct $f$ and $|g|$ from the final-time data.

In this work, we consider the ISP for the stochastic wave equation where the source is driven by the fBm. It contains three contributions. First, we show that the direct problem admits a well-defined mild solution which satisfies some stability estimate. Second, the uniqueness and instability are discussed for the inverse problem, which is to use the empirical expectation and correlation of the final-time data $u(x,T)$ to reconstruct $f$ and $|g|$ of the source term $F$. In \cite{FLW2020}, the ISP was studied for the time fractional diffusion equation driven by a fBm, where the parameter of the Caputo fractional derivative $\alpha$ is restricted to $0<\alpha\leq1$. Some of the results presented in \cite{FLW2020} are not valid any more for the hyperbolic equation where $\alpha=2$. In particular, the uniqueness of the inverse problem could not be guaranteed. Therefore it is worthwhile to investigate separately the ISP for the stochastic wave equation. Third, the numerical experiments are given to illustrate 
how to obtain $f$ and $|g|$. Since the ISP is ill-posed, the truncation based regularization method is adopted to reconstruct $f$ and $|g|$. The numerical examples show that the method is effective to recover the source functions.

The paper is organized as follows. In Section 2, we introduce some preliminaries for the fBm and the mild solution of the stochastic wave equation. Section 3 is concerned with the well-posedness of the direct problem. Section 4 is devoted to the inverse problem. The uniqueness and instability are discussed. The numerical experiments are presented in Section 5. The paper concludes with some general remarks in Section 6.

\section{Preliminaries}

We begin with a brief introduction to fBms. The details can be found in \cite{Nualart+2006,FLW2020}. Let $(\Omega, \mathcal F, \mathbb{P})$ be a complete probability space, where $\Omega$ is a sample space, $\mathcal F$ is a $\sigma$-algebra on $\Omega$, and $\mathbb{P}$ is a probability measure on the measurable space $(\Omega, \mathcal F)$. For a random variable $X$, denote by $\mathbb{E}(X)$ and $\mathbb{V}(X)=\mathbb{E}(X-\mathbb{E}(X))^2 = \mathbb{E}(X^2)-(\mathbb{E}(X))^2$ the expectation and variance of $X$, respectively. For two random variables $X$ and $Y$, $\text{Cov}(X,Y)=\mathbb{E}[(X-\mathbb{E}(X))(Y-\mathbb{E}(Y))]$ stands for the covariance of $X$ and $Y$. In the sequel, the dependence of random variables on the sample $\omega\in\Omega$ will be omitted unless it is necessary to avoid confusion.

The one-dimensional fBm $B^H(t)$ is a centered Gaussian process, which satisfies $B^H(0)=0$ and is determined by the covariance function
\begin{equation*}
R(t,s)=\mathbb{E}[B^H(t)B^H(s)]=\frac{1}{2}\left(t^{2H}+s^{2H}-|t-s|^{2H}\right)
\end{equation*}
for any $s,t\ge0$. In particular, if $H=\frac12$, $B^H$ turns to be the standard Brownian motion, which is usually denoted by $W$ and has the covariance function $R(t,s)=t\wedge s$.

The increments of fBms satisfy
\begin{equation*}
\mathbb{E}\left[\left(B^H(t)-B^H(s)\right)\left(B^H(s)-B^H(r)\right)\right]=\frac{1}{2}\left[(t-r)^{2H}-(t-s)^{2H}-(s-r)^{2H}\right]
\end{equation*}
and
\begin{equation*}
\mathbb{E}\left[\left(B^H(t)-B^H(s)\right)^2\right]=(t-s)^{2H}
\end{equation*}
for any $0<r<s<t$. It indicates that the increments of $B^H(t)$ in disjoint intervals are linearly dependent except for the case $H=\frac{1}{2}$, and the increments are stationary since their moments depend only on the length of the interval.

Based on the moment estimates and the Kolmogorov continuity criterion, it holds
for any $\epsilon>0$ and $s,t\in[0,T]$ that
\[
|B^H(t)-B^H(s)|\le C|t-s|^{H-\epsilon}
\]
almost surely with constant $C$ depending on $\epsilon$ and $T$. It is clear to note that $H$ represents the regularity of $B^H$ and the trajectories of $B^H$ are $(H-\epsilon)$-H\"older continuous.

The fBm $B^H$ with $H\in(0,1)$ has a Wiener integral representation
\[
B^H(t)=\int_0^tK_H(t,s)dW(s),
\]
where $K_H$ is a square integrable kernel and $W$ is the standard Brownian motion.

For a fixed interval $[0,T]$, denote by $\mathcal{E}$ the space of step functions on $[0,T]$ and by $\mathcal{H}$ the closure of $\mathcal{E}$ with respect to the product
\[
 \langle\chi_{[0, t]}, \chi_{[0, s]}\rangle_{\mathcal H}=R(t, s),
\]
where $\chi_{[0, t]}$ and $\chi_{[0, s]}$ are the characteristic functions. For $\psi(t),\phi(t)\in\mathcal{H}$, it follows from the It\^{o} isometry that

(1) if $H\in(0,\frac12)$,
\begin{eqnarray}\label{E012}
&&\mathbb{E}\left[\int_0^t\psi(s)dB^H(s)\int_0^t\phi(s)dB^H(s)\right]\nonumber\\
&=&\int_0^t\bigg\{c_H\Big[\big(\frac{t}{s}\big)^{H-\frac12}(t-s)^{H-\frac12}-(H-\frac12)s^{\frac12-H}\int_s^tu^{H-\frac32}(u-s)^{H-\frac12}du\Big]\psi(s)\nonumber\\
&& +\int_s^t(\psi(u)-\psi(s))c_H\big(\frac{u}{s}\big)^{H-\frac12}(u-s)^{H-\frac32}du\bigg\}\nonumber\\
&&\times \bigg\{c_H\left[\big(\frac{t}{s}\big)^{H-\frac12}(t-s)^{H-\frac12}-(H-\frac12)s^{\frac12-H}\int_s^tu^{H-\frac32}(u-s)^{H-\frac12}du\right]\phi(s)\nonumber\\
&& +\int_s^t(\phi(u)-\phi(s))c_H\big(\frac{u}{s}\big)^{H-\frac12}(u-s)^{H-\frac32}du\bigg\}ds,
\end{eqnarray}
where $c_H=\left(\frac{2H}{(1-2H)\beta(1-2H,H+\frac12)}\right)^{\frac12}$ and $\beta(p,q)=\int_0^1t^{p-1}(1-t)^{q-1}dt$;

(2) if $H=\frac12$,
\begin{align}\label{E12}
\mathbb{E}\left[\int_0^t\psi(s)dB^H(s)\int_0^t\phi(s)dB^H(s)\right]=\mathbb{E}\left[\int_0^t\psi(s)dW(s)\int_0^t\phi(s)dW(s)\right]=\int_0^t\psi(s)\phi(s)ds;
\end{align}

(3) if $H\in(\frac12,1)$,
\begin{align}\label{E121}
\mathbb{E}\left[\int_0^t\psi(s)dB^H(s)\int_0^t\phi(s)dB^H(s)\right]=\alpha_H\int_0^t\int_0^t\psi(r)\phi(u)|r-u|^{2H-2}dudr,
\end{align}
where $\alpha_H=H(2H-1)$.

Since $\dot{B}^H$ is a distribution instead of a classical function, the stochastic wave equation in \eqref{directp} does not hold pointwisely; it should be interpreted as an integral equation and its mild solution is defined as follows.

\begin{definition}
A stochastic process $u$ taking values in $L^2(D)$ is called a mild solution of \eqref{directp} if
\[
u(x,t)=\int_0^t\sin\left((t-\tau)(-\Delta)^{\frac12}\right)(-\Delta)^{-\frac12}f(x)h(\tau)d\tau+\int_0^t\sin\left((t-\tau)(-\Delta)^{\frac12}\right)(-\Delta)^{-\frac12}g(x)dB^H(\tau)
\]
is well-defined almost surely.
\end{definition}

It is known that the operator $-\Delta$ with the homogeneous Dirichlet boundary condition has an eigensystem $\{\lambda_k, \varphi_k\}_{k=1}^{\infty}$, where the eigenvalues satisfy $0<\lambda_1\leq\lambda_2\leq\cdots\leq\lambda_k\leq\cdots$ with $
\lambda_k\to\infty$ as $k\to\infty$, and the eigen-functions $\{\varphi_k\}_{k=1}^{\infty}$ form a complete and orthonormal basis for $L^2({D})$. For any function $v$ in $L^2(D)$, it can be written as
\[
 v(x)=\sum_{k=1}^{\infty}v_k\varphi_k(x), \quad v_k=(v,\varphi_k)_{L^2(D)}=\int_D v(x)\varphi_k(x)dx.
\]
Hence, if $u\in L^2(D)$ is a mild solution of \eqref{directp}, we have
\begin{equation}\label{Dsolution}
u(\cdot,t)=\sum_{k=1}^{\infty}u_k(t)\varphi_k,
\end{equation}
where
\begin{eqnarray}\label{Ik}
u_k(t)&=&(u(\cdot,t),\varphi_k)_{L^2(D)}\nonumber\\
&=&f_k\int_0^t\frac{\sin((t-\tau)\sqrt{\lambda_k})}{\sqrt{\lambda_k}}h(\tau)d\tau+g_k\int_0^t\frac{\sin((t-\tau)\sqrt{\lambda_k})}{\sqrt{\lambda_k}}dB^H(\tau)\nonumber\\
&=:&I_{k,1}(t)+I_{k,2}(t).
\end{eqnarray}
Here $f_k=(f,\varphi_k)_{L^2({D})}$, $g_k=(g,\varphi_k)_{L^2({D})}$, and $u_k(t)$ satisfies the stochastic differential equation
 \begin{equation}\label{SODE}
\begin{cases}
u_k''(t)+ \lambda_k u_k(t)=f_kh(t)+g_k\dot{B}^H(t), \quad  t\in(0,T),\\
u_k(0)=u_k'(0)=0.
\end{cases}
\end{equation}
In particular, if $g=0$, the stochastic differential equation \eqref{SODE} reduces to the deterministic differential
equation
\begin{equation*}
\begin{cases}
u_k''(t)+ \lambda_k u_k(t)=f_kh(t), \quad  t\in(0,T),\\
u_k(0)=u_k'(0)=0,
\end{cases}
\end{equation*}
which has the solution given as $I_{k,1}(t)$ in (\ref{Ik}).

\section{The direct problem}

In this section, we discuss the well-posedness of the direct problem. It is only necessary to address the stability since the existence and the uniqueness of the solution has already been considered (cf. \cite{Orsingher+1982,Walsh+1984,Tindel+2007}). We show that the mild solution \eqref{Dsolution} of the initial-boundary value problem \eqref{directp} is well-defined under the following assumptions.

\begin{assumption}\label{assumption}
Let $H\in(0,1)$ and $f,g\in L^2({D})$ with $\|g\|_{L^2(D)}\neq0$. Assume in addition that $h\in L^{\infty}(0,T)$ is a nonnegative function and its support has a positive measure.
\end{assumption}

It is easy to note that the mild solution (\ref{Dsolution}) satisfies
\begin{eqnarray*}
\|u(\cdot,t)\|^2_{L^2({D})} &=&\left\|\sum_{k=1}^{\infty}(I_{k,1}(t)+I_{k,2}(t))\varphi_k(\cdot)
\right\|_{L^2({D})}^2=\sum_{k=1}^{\infty}(I_{k,1}(t)+I_{k,2}(t))^2 \\
&\lesssim& \sum_{k=1}^{\infty} I^2_{k,1}(t)+\sum_{k=1}^{\infty}I_{k,2}^2(t),
\end{eqnarray*}
which gives
\begin{eqnarray}\label{ex}
\mathbb{E}\left[\|u\|^2_{L^2({D}\times[0,T])}\right] & =& \mathbb{E}\left[\int_0^T \|u(\cdot,t)\|^2_{L^2({D})}dt\right]\nonumber\\
&\lesssim&\mathbb{E}\bigg[\int_0^T \bigg(\sum_{k=1}^{\infty}I^2_{k,1}(t)+\sum_{k=1}^{\infty}I_{k,2}^2(t)\bigg)dt\bigg]\nonumber\\
&=&\int_0^T \bigg(\sum_{k=1}^{\infty}I^2_{k,1}(t)\bigg)dt+\mathbb{E}\bigg[\int_0^T \bigg(\sum_{k=1}^{\infty}I_{k,2}^2(t)\bigg)dt\bigg]\nonumber\\
&=&\sum_{k=1}^{\infty}\|I_{k,1}\|^2_{L^2(0,T)}+\int_0^T \bigg(\sum_{k=1}^{\infty}\mathbb{E}[I_{k,2}^2(t)]
\bigg)dt\nonumber\\
&=:&S_1+S_2.
\end{eqnarray}
Hereinafter, $a\lesssim b$ stands for $a\leqslant Cb$, where $C>0$ is a constant and its specific value is not required but should be clear from the context.

Let $G_{k}(t)=\sin(\sqrt{\lambda_k}t)/\sqrt{\lambda_k}$. Then
\begin{eqnarray}\label{Ik11}
\|I_{k,1}\|_{L^2(0,T)}&=&\left(\int_0^T\left|f_k\int_0^tG_k(\tau)h(t-\tau)d\tau\right|^2dt\right)^{\frac12}\notag\\
&\leq&\left(\int_0^T|f_k|^2\|h\|_{L^{\infty}(0,T)}^2\left(\int_0^T|G_k(\tau)|d\tau\right)^2dt\right)^{\frac12}\notag\\
&\leq& T^{\frac12}|f_k|\|G_{k}\|_{L^1(0,T)}\|h\|_{L^\infty(0,T)}.
\end{eqnarray}
A simple calculation gives
\begin{equation}\label{G}
\|G_{k}\|_{L^1(0,T)}=\int_0^T\left|\frac{\sin(\sqrt{\lambda_k}t)}{\sqrt{\lambda_k}}\right|dt\leq\int_0^T tdt=\frac{T^2}{2}.
\end{equation}
Combining (\ref{Ik11}) and (\ref{G}), we obtain
\begin{equation}\label{S1}
S_1=\sum_{k=1}^{\infty}\|I_{k,1}\|^2_{L^2(0,T)}
\leq\frac{T^{5}}{4}\sum_{k=1}^{\infty}|f_k|^2\|h\|^2_{L^\infty(0,T)}
=\frac{T^{5}}{4}\|h\|^2_{L^\infty(0,T)}\|f\|^2_{L^2({D})}.
\end{equation}

Next is to estimate $S_2$. It follows from \eqref{Ik} that
\begin{align}\label{exIk2}
\mathbb{E}[I_{k,2}^2(t)]=g_k^2\mathbb{E}\Big[\Big(\int_0^t\frac{\sin(\sqrt{\lambda_k}(t-\tau))}{\sqrt{\lambda_k}}dB^H(\tau)\Big)^2\Big].
\end{align}
Below, we discuss separately the cases $H=\frac{1}{2}$, $H\in(0,\frac{1}{2})$, and $H\in(\frac{1}{2},1)$  since the covariance operator of $B^H$ takes different forms in these three cases.

For the case $H=\frac{1}{2}$, it follows from It\^o's isometry \eqref{E12} that
\begin{align*}
\mathbb{E}\Big[\Big(\int_0^t\frac{\sin(\sqrt{\lambda_k}(t-\tau))}{\sqrt{\lambda_k}}dB^{\frac{1}{2}}(\tau)\Big)^2\Big]
=\int_0^t\frac{\sin^2(\sqrt{\lambda_k}(t-\tau))}{\lambda_k}d\tau\leq\int_0^t(t-\tau)^2d\tau=\frac{t^3}{3}.
\end{align*}

For the case $H\in(0,\frac{1}{2})$, we have from (\ref{E012}) that
\begin{eqnarray}\label{Ee012}
&&\mathbb{E}\Big[
\Big(\int_0^t\frac{\sin(\sqrt{\lambda_k}(t-\tau))}{\sqrt{\lambda_k}}dB^H(\tau)\Big)^2\Big]\nonumber\\
&=&\int_0^t \bigg\{c_H\Big[\big(\frac{t}{\tau}\big)^{H-\frac12}(t-\tau)^{H-\frac12}-(H-\frac12)\tau^{\frac12-H}\int_\tau^tu^{H-\frac32}(u-\tau)^{H-\frac12}du\Big]\frac{\sin(\sqrt{\lambda_k}(t-\tau))}{\sqrt{\lambda_k}}\nonumber\\
&&+\int_\tau^t\Big(\frac{\sin(\sqrt{\lambda_k}(t-u))}{\sqrt{\lambda_k}}-\frac{\sin(\sqrt{\lambda_k}(t-\tau))}{\sqrt{\lambda_k}}\Big)c_H\big(\frac{u}{\tau}\big)^{H-\frac12}(H-\frac12)(u-\tau)^{H-\frac32}du\bigg\}^2d\tau\nonumber\\
&\lesssim&\int_0^t\Big[\big(\frac{t}{\tau}\big)^{H-\frac12}(t-\tau)^{H-\frac12}\frac{\sin(\sqrt{\lambda_k}(t-\tau))}{\sqrt{\lambda_k}}\Big]^2d\tau\notag\\
&&+\int_0^t\Big[\tau^{\frac12-H}\Big(\int_\tau^tu^{H-\frac32}(u-\tau)^{H-\frac12}du\Big)\frac{\sin(\sqrt{\lambda_k}(t-\tau))}{\sqrt{\lambda_k}}\Big]^2d\tau\nonumber\\
&&+\int_0^t\Big[\int_\tau^t\Big(\frac{\sin(\sqrt{\lambda_k}(t-u))}{\sqrt{\lambda_k}}-\frac{\sin(\sqrt{\lambda_k}(t-\tau))}{\sqrt{\lambda_k}}\Big)\big(\frac{u}{\tau}\big)^{H-\frac12}(u-\tau)^{H-\frac32}du\Big]^2d\tau\nonumber\\
&=:&I_1(t)+I_2(t)+I_3(t).
\end{eqnarray}
A simple calculation gives
\begin{eqnarray}\label{I1}
I_1(t)&=&\int_0^t\big(\frac{t}{\tau}\big)^{2H-1}(t-\tau)^{2H-1}\frac{\sin^2(\sqrt{\lambda_k}(t-\tau))}{\lambda_k}d\tau\nonumber\\
&\leq&\int_0^t(\frac{t}{\tau})^{2H-1}(t-\tau)^{2H-1}(t-\tau)^2d\tau\nonumber\\
&\leq&\int_0^t(t-\tau)^{2H+1}d\tau=\frac{t^{2H+2}}{2H+2}.
\end{eqnarray}
Similarly, we have
\begin{align}\label{I21}
I_2(t)&\lesssim\int_0^t\tau^{1-2H}(t-\tau)^2\Big(\int_\tau^tu^{H-\frac32}(u-\tau)^{H-\frac12}du\Big)^2d\tau.
\end{align}
Using the binomial expansion leads to
\begin{eqnarray*}
\int_\tau^tu^{H-\frac32}(u-\tau)^{H-\frac12}du
&=&\int_\tau^t u^{2H-2}\bigg[\sum_{n=0}^{\infty}
\begin{pmatrix}
H-\frac12\\
n
\end{pmatrix}
\big(-\frac{\tau}{u}\big)^n\bigg]du \\
&=&\sum_{n=0}^{\infty}
\begin{pmatrix}
H-\frac12\\
n
\end{pmatrix}(-1)^n\tau^n\left(\frac{t^{2H-1-n}-\tau^{2H-1-n}}{2H-1-n}\right)\\
&\leq&(t^{2H-1}-\tau^{2H-1})\sum_{n=0}^{\infty}
\begin{pmatrix}
H-\frac12\\
n
\end{pmatrix}\frac{(-1)^n}{2H-1-n}\\
&\lesssim& t^{2H-1}-\tau^{2H-1}.
\end{eqnarray*}
Combing the above estimates, we obtain from (\ref{I21}) that
\begin{eqnarray}\label{I2}
I_2(t)&\lesssim&\int_0^t\tau^{1-2H}(t-\tau)^2\left(t^{2H-1}-\tau^{2H-1}\right)^2d\tau\nonumber\\
&\lesssim&\int_0^t\tau^{1-2H}(t-\tau)^2\left(t^{4H-2}+\tau^{4H-2}\right)d\tau\nonumber\\
&\lesssim&\int_0^t\tau^{2H-1}(t-\tau)^2d\tau\nonumber\\
&\leq&\int_0^t\tau^{2H-1}t^2d\tau
\lesssim t^{2H+2}.
\end{eqnarray}
By the mean value theorem, there holds
\begin{eqnarray}\label{I3}
I_3(t)&\leq&\int_0^t\Big[\int_\tau^t|\tau-u|(\frac{u}{\tau})^{H-\frac12}(u-\tau)^{H-\frac32}du\Big]^2d\tau\nonumber\\
&\leq&\int_0^t\Big[\int_\tau^t(u-\tau)^{H-\frac12}du\Big]^2d\tau \lesssim t^{2H+2}.
\end{eqnarray}
Substituting (\ref{I1}), (\ref{I2}) and (\ref{I3}) into (\ref{Ee012}), we obtain for $H\in(0,\frac12)$ that
\begin{align*}
\mathbb{E}\Big[
\Big(\int_0^t\frac{\sin(\sqrt{\lambda_k}(t-\tau))}{\sqrt{\lambda_k}}dB^H(\tau)\Big)^2\Big]\lesssim t^{2H+2}.
\end{align*}

For the case $H\in(\frac{1}{2},1)$, we have from \eqref{E121} that
\begin{eqnarray*}
&&\mathbb{E}\Big[
\Big(\int_0^t\frac{\sin(\sqrt{\lambda_k}(t-\tau))}{\sqrt{\lambda_k}}dB^H(\tau)\Big)^2\Big]\\
&=&\alpha_H\int_0^t\int_0^t\frac{\sin(\sqrt{\lambda_k}(t-p))}{\sqrt{\lambda_k}}\frac{\sin(\sqrt{\lambda_k}(t-q))}{\sqrt{\lambda_k}}|p-q|^{2H-2}dpdq\\
&\leq&\alpha_H\int_0^t\int_0^t(t-p)(t-q)|p-q|^{2H-2}dpdq\\
&=&2\alpha_H\int_0^t\int_q^tpq(p-q)^{2H-2}dpdq,\\
&=&2\alpha_H\int_0^tq\frac{(t-q)^{2H-1}}{2H-1}\big(t-\frac{t-q}{2H}\big)dq\\
&\lesssim& t\int_0^tq\frac{(t-q)^{2H-1}}{2H-1}dq
\lesssim t^{2H+2}.
\end{eqnarray*}

Combing the above estimates, we obtain for any $H\in (0,1)$ that
\begin{align}\label{eq:EeH01}
\mathbb{E}\Big[\Big(\int_0^t\frac{\sin(\sqrt{\lambda_k}(t-\tau))}{\sqrt{\lambda_k}}dB^H(\tau)\Big)^2\Big]
\lesssim t^{2H+2},
\end{align}
which implies the following stability estimate for the mild solution (\ref{Dsolution}).

\begin{theorem}\label{thereom1}
Let Assumption \ref{assumption} hold. Then the stochastic process $u$
given in (\ref{Dsolution}) satisfies
\begin{equation*}
\mathbb{E}[\|u\|^2_{L^2({D}\times[0,T])}]\lesssim\frac{T^5}{4}\|h\|^{2}_{L^\infty(0,T)}\|f\|^
{2}_{L^2({D})}+\frac{T^{2H+3}}{2H+3}\|g\|^{2}_{L^2({D})}.
\end{equation*}
\end{theorem}

\begin{proof}
The proof follows easily from (\ref{ex}), (\ref{S1}), (\ref{exIk2}), and
(\ref{eq:EeH01}).
\end{proof}

\begin{remark}
The above stability estimate is consistent with \cite[Theorem 3.1]{FLW2020}. The stochastic wave equation (\ref{directp}) can be viewed as the case $\alpha=2$ of the problem (2.1) in \cite{FLW2020}, where
\[
E_{2,2}(z^2)=\frac{\sinh(z)}{z}
\]
and
\[
(t-\tau)^{\alpha-1}E_{\alpha,\alpha}\big(-(t-\tau)^\alpha\lambda_k\big)\big|_{\alpha=2}=(t-\tau)\frac{\sinh\sqrt{-(t-\tau)^2\lambda_k}}{\sqrt{-(t-\tau)^2\lambda_k}}=\frac{\sin(\sqrt{\lambda_k}(t-\tau))}{\sqrt{\lambda_k}}.
\]
\end{remark}

\section{The inverse problem}

In this section, we consider the inverse problem of reconstructing $f$ and $|g|$ from the empirical expectation and correlation of the final-time data $u(x,T)$. More specifically, the data is assumed to be given by
\[
 u_k(T):=(u(\cdot,T),\varphi_k(\cdot))_{L^2(D)}.
\]
We discuss the uniqueness and the issue of instability for the inverse problem, separately.

\subsection{Uniqueness}

It follows from (\ref{Dsolution}), (\ref{Ik})  and \eqref{exIk2} that
\begin{equation}\label{Iexp}
\mathbb{E}(u_k(T))=f_k\int_0^Th(\tau)\frac{\sin(\sqrt{\lambda_k}(T-\tau))}{\sqrt{\lambda_k}}d\tau
\end{equation}
and
\begin{align}\label{Icvar}
\text{Cov}(u_k(T),u_l(T))
=g_kg_l
\mathbb{E}\Big[\int_0^T\frac{\sin(\sqrt{\lambda_k}(T-\tau))}{\sqrt{\lambda_k}}dB^H(\tau)\int_0^T\frac{\sin(\sqrt{\lambda_l}(T-\tau))}{\sqrt{\lambda_l}}dB^H(\tau)\Big].
\end{align}

First, we consider the reconstruction of $f$. Let
\begin{align*}
h_k=\int_0^Th(\tau)\sin(\sqrt{\lambda_k}(T-\tau))d\tau=\int_0^Th(T-\tau)\sin(\sqrt{\lambda_k}\tau)d\tau.
\end{align*}
By (\ref{Iexp}), it is impossible to compute $f_k$ if $h_k=0$, which indicates nonuniqueness to reconstruct $f$. It can be seen that $h_k$ may be zero for some $h(t)$ satisfying Assumption \ref{assumption}. Clearly, $\sin(\sqrt{\lambda_k}t)$ is an oscillatory function; the larger the $\lambda_k$ is, the more oscillatory the function $\sin(\sqrt{\lambda_k}t)$ is. Hence the uniqueness cannot be guaranteed to reconstruct $f$ for a general $h$. Below we consider two special cases for the function $h$ where it is sufficient to ensure the uniqueness for the inverse problem.

Case $1$: $h(t)\equiv C>0$ is a constant function and $D=(0,\pi)$. In this case, it is clear that the eigen-system of the Laplacian operator $-\Delta$ is
\[
\lambda_k=k^2, \quad \varphi_k=\sqrt{\frac2\pi}\sin(kx), \quad k=1,2,...
\]
A simple calculation yields
\begin{align*}
h_k~=C\int_0^T\sin(k\tau)d\tau~=\frac Ck(1-\cos(kT)).
\end{align*}
Note that the final time $T$ can be chosen arbitrarily. For example, we can choose  $T$ as a rational number so that $h_k>0$ for any $k\in \mathbb{N}$.

Case $2$: $h(t)$ is a continuous and monotonously increasing function and $D=(0,\pi)$. Divide the time interval $(0,T]$ into $n+1$ subintervals as $(t_0,t_1],\dots$, $(t_i,t_{i+1}],\dots$, $(t_n,t_{n+1}]$, where $kt_i=i\pi$, $i=0,1,2,...,n$, and $t_0=0$, $t_{n+1}=T$. Since $h(t)$ is continuous and $\sin(kt)$ is integrable and does not change sign on each subinterval $(t_{i},t_{i+1}]$, by the mean value theorem for the definite integral, there exists $\xi_i\in(t_{i},t_{i+1}]$ such that
\begin{eqnarray*}
h_k&=&\int_0^Th(T-\tau)\sin(k\tau)d\tau=\sum_{i=0}^{n}h(T-\xi_i)\int_{t_i}^{t_i+1}\sin(k\tau)d\tau\nonumber\\
&=&\sum_{i=0}^{n-1}h(T-\xi_i)\frac{2\cdot(-1)^i}{k}+h(T-\xi_n)\Big(\frac{(-1)^n-\cos(kT)}{k}\Big).
\end{eqnarray*}
If $n$ is even, since $h(t)$ is monotonously increasing,
\begin{align*}
h_k=\sum_{i=2m,m=0}^{\frac{n-2}{2}}(h(T-\xi_i)-h(T-\xi_{i+1}))\frac2k+h(T-\xi_n)\Big(\frac{1-\cos(kT)}{k}\Big)>0.
\end{align*}
If $n$ is odd, similarly, we have
\begin{align*}
h_k=\sum_{i=2m,m=0}^{\frac{n-3}{2}}(h(T-\xi_i)-h(T-\xi_{i+1}))\frac2k+h(T-\xi_{n-1})\frac{2}{k}+h(T-\xi_n)\Big(\frac{-1-\cos(kT)}{k}\Big)>0.
\end{align*}
Here, the final time $T$ is chosen as a rational number. If $h(t)$ is strictly monotonously increasing, $T$ can be any real number.

\begin{remark}
Although we only consider the one-dimensional case $D=(0,\pi)$, in fact, if $D$ is a bounded domain in $\mathbb{R}^d, d\in\{1,2,3\}$, then the $j$-th eigenvalue $\lambda_j$ of the homogeneous Dirichlet boundary problem for the Laplacian operator $-\Delta$ in $D$ satisfies
\begin{equation*}
c_0 j^{\frac2d}\leq\lambda_j\leq c_1 j^{\frac2d},
\end{equation*}
where $j\in \mathbb{N}$, and the constants $c_0, c_1$ are independent of the index $j$ \cite{Dengnumerical+2020,Strauss+2008}. Therefore, if we choose $T$ properly, the uniqueness can still be obtained to recover $f$ when the function $h(t)$ belongs to the two cases considered above.
\end{remark}

\begin{remark}
From (\ref{Iexp}), if $f_k\neq0, D=(0,\pi)$, then there holds
\begin{equation}\label{Iexp1}
\frac{k\mathbb{E}(u_k(T))}{f_k}=\int_0^Th(T-\tau)\sin(k\tau)d\tau.
\end{equation}
The right hand side of (\ref{Iexp1}) can be understood as the sine transform of $h(T-\tau)$. Alternatively, (\ref{Iexp1}) can be written as
\begin{equation}\label{Iexp2}
k\mathbb{E}(u_k(T))=\sqrt{\frac2\pi}\int_0^T\int_0^\pi h(T-\tau)f(x)\sin(kx)\sin(k\tau)dxd\tau.
\end{equation}
Therefore, $h(T-t)f(x)$ can be reconstructed by taking the inverse two-dimensional sine transform on both sides of (\ref{Iexp2}) \cite{Bao+Chow+Li+Zhou+2014}.
\end{remark}

Next, we consider the reconstruction of $|g|$. Let
\begin{align}\label{guniqbm}
E_{kl}:=\mathbb{E}\Big[\int_0^T\frac{\sin(\sqrt{\lambda_k}(T-\tau))}{\sqrt{\lambda_k}}dB^H(\tau)\int_0^T\frac{\sin(\sqrt{\lambda_l}(T-\tau))}{\sqrt{\lambda_l}}dB^H(\tau)\Big].
\end{align}
Similarly, since $\sin(\sqrt{\lambda_k}t)$ is an oscillatory function, and the larger the $\lambda_k$ is the more oscillatory the function $\sin(\sqrt{\lambda_k}t)$ is, we can see from (\ref{guniqbm}) that it is difficult to guarantee $E_{kl}\neq0$, particularly for a general $H\in(0, 1)$. Below, we only consider the case $H=\frac{1}{2}$ and $D=(0,\pi) $, where the uniqueness can be obtained to recover $|g|$.

By (\ref{E12}), a straightforward calculation gives
\begin{eqnarray}\label{Ekl12}
E_{kl}&=&\int_0^T\frac{\sin(\sqrt{\lambda_k}(T-\tau))}{\sqrt{\lambda_k}}\frac{\sin(\sqrt{\lambda_l}(T-\tau))}{\sqrt{\lambda_l}}d\tau\nonumber\\
&=&\begin{cases}  \frac{1}{2\lambda_k}\left(T-\frac{\sin(2\sqrt{\lambda_k}T)}{2\sqrt{\lambda_k}}\right)>0,\quad & k=l,\\
\frac{1}{\sqrt{\lambda_k\lambda_l}}\bigg(-\frac{\sin((\sqrt{\lambda_k}+\sqrt{\lambda_l})T)}{2(\sqrt{\lambda_k}+\sqrt{\lambda_l})}+\frac{\sin((\sqrt{\lambda_k}-\sqrt{\lambda_l})T)}{2(\sqrt{\lambda_k}-\sqrt{\lambda_l})}\bigg),\quad & k\neq l.
\end{cases}
\end{eqnarray}
When $D=(0,\pi)$, and $k\neq l$, we have
\begin{align*}
E_{kl}=\frac{1}{kl}\bigg(-\frac{\sin((k+l)T)}{2(k+l)}+\frac{\sin((k-l)T)}{2(k-l)}\bigg).
\end{align*}
Without loss of generality, let $k>l$, then it holds that $E_{kl}\neq0$ for any algebraic number $T$. Otherwise, if $E_{kl}=0$, then
\[
\frac{\sin((k-l)T)}{2(k-l)}-\frac{\sin((k+l)T)}{2(k+l)}=0,
\]
which implies
\begin{equation*}
(k+l)e^{i(k-l)T}-(k+l)e^{-i(k-l)T}-(k-l)e^{i(k+l)T}+(k-l)e^{-i(k+l)T}=0.
\end{equation*}
Since $\{i(k-l)T, -i(k-l)T, i(k+l)T, -i(k+l)T\}$ is a set of distinct algebraic numbers, by the Lindemann--Weierstrass theorem (cf. \cite[Theorem 1.4]{ab90}), we have for any non-zero algebraic numbers $c_1, c_2, c_3, c_4$ that
\begin{equation*}
c_1e^{i(k-l)T}+c_2e^{-i(k-l)T}+c_3e^{i(k+l)T}+c_4e^{-i(k+l)T}\neq0,
\end{equation*}
which leads to a contradiction. Therefore $E_{kl}\neq0$ for any $k\neq l\in \mathbb{N}$. Combining with (\ref{Ekl12}), we deduce for any $k,l\in\mathbb{N}$ that $E_{kl}\neq0$ for any algebraic number $T$.

The following results concern the statement of uniqueness for the inverse problem.

\begin{theorem}\label{thereom2}
Let Assumption \ref{assumption} hold.
\begin{itemize}
  \item [(1)]If $h(t)$ is monotonously increasing and $T$ is a rational number, or if $h(t)$ is strictly monotonously increasing and $T$ is a real number. Then $f$ can be uniquely determined by the data set
$\{\mathbb{E}(u_k(T)): k \in\mathbb{N}\}$.
  \item [(2)]If T is any algebraic number, then the source term $g$ up to sign, i.e. $\pm g$, can be uniquely determined by the data set $\{\text{Cov}(u_k(T),u_l(T)): k,l \in \mathbb{N}\}$.
\end{itemize}
\end{theorem}

\subsection{Instability}

In this subsection, we demonstrate the inverse problem is unstable to recover $f$ and $|g|$. First, it is clear to note that
\begin{equation*}
\left|\int_0^Th(\tau)\frac{\sin(\sqrt{\lambda_k}(T-\tau))}{\sqrt{\lambda_k}}d\tau\right|\leq\frac{1}{\sqrt{\lambda_k}}\int_0^Th(\tau)d\tau\rightarrow0 ~ \text{as } k\rightarrow\infty,
\end{equation*}
which shows that it is unstable to recover $f$ due to (\ref{Iexp}).

Next, we discuss the instability of recovering $g^2_k$, which is equivalent to the instability of recovering $|g|$. We discuss the three different cases  $H=\frac{1}{2}, H\in(0,\frac{1}{2})$ and $H\in(\frac{1}{2},1)$, separately.

For the case $H=\frac{1}{2}$, by (\ref{E12}), there holds
\begin{align*}
\mathbb{E}\Big[
\Big(\int_0^T\frac{\sin(\sqrt{\lambda_k}(T-\tau))}{\sqrt{\lambda_k}}dB^{\frac{1}{2}}(\tau)\Big)^2\Big]\leq\frac{T}{\lambda_k}.
\end{align*}

For the case $H\in(0,\frac{1}{2})$, we consider (\ref{Ee012}) with $t=T$ and the estimate of $I_j(T), j=1,2,3.$
For $I_1(T)$, a simple calculation yields
\begin{eqnarray*}
I_1(T)&=&\int_0^T(\frac{T}{\tau})^{2H-1}(T-\tau)^{2H-1}\frac{\sin^2(\sqrt{\lambda_k}(T-\tau))}{\lambda_k}d\tau\nonumber\\
&\leq&\frac{1}{\lambda_k}\int_0^T(\frac{T}{\tau})^{2H-1}(T-\tau)^{2H-1}d\tau
\leq\frac{1}{\lambda_k}\frac{T^{2H}}{2H}.
\end{eqnarray*}
About $I_2(T)$, we have
\begin{eqnarray*}
I_2(T)&\lesssim&\frac{1}{\lambda_k}\int_0^T\tau^{1-2H}\Big(\int_\tau^Tu^{H-\frac32}(u-\tau)^{H-\frac12}du\Big)^2d\tau\nonumber\\
&\lesssim&\frac{1}{\lambda_k}\int_0^T\tau^{1-2H}(T^{2H-1}-\tau^{2H-1})^2d\tau\nonumber\\
&\lesssim&\frac{1}{\lambda_k}\int_0^T\tau^{1-2H}(T^{4H-2}+\tau^{4H-2})d\tau\nonumber\\
&\lesssim&\frac{1}{\lambda_k}\int_0^T\tau^{2H-1}d\tau
\lesssim\frac{1}{\lambda_k}\frac{T^{2H}}{2H}.
\end{eqnarray*}
About $I_3(T)$, there holds
\begin{eqnarray*}
I_3(T)&=&\int_0^T\Big[\int_\tau^T(\frac{\sin(\sqrt{\lambda_k}(T-u))}{\sqrt{\lambda_k}}-\frac{\sin(\sqrt{\lambda_k}(T-\tau))}{\sqrt{\lambda_k}})(\frac{u}{\tau})^{H-\frac12}(u-\tau)^{H-\frac32}du\Big]^2d\tau\\
&\lesssim&\int_0^{t_*}\Big[\int_\tau^{t_*}(\frac{\sin(\sqrt{\lambda_k}(T-u))}{\sqrt{\lambda_k}}-\frac{\sin(\sqrt{\lambda_k}(T-\tau))}{\sqrt{\lambda_k}})(\frac{u}{\tau})^{H-\frac12}(u-\tau)^{H-\frac32}du\Big]^2d\tau\\
&&+\int_0^{t_*}\Big[\int_{t_*}^T(\frac{\sin(\sqrt{\lambda_k}(T-u))}{\sqrt{\lambda_k}}-\frac{\sin(\sqrt{\lambda_k}(T-\tau))}{\sqrt{\lambda_k}})(\frac{u}{\tau})^{H-\frac12}(u-\tau)^{H-\frac32}du\Big]^2d\tau\\
&&+\int_{t_*}^T\Big[\int_\tau^{\tau+t_*}(\frac{\sin(\sqrt{\lambda_k}(T-u))}{\sqrt{\lambda_k}}-\frac{\sin(\sqrt{\lambda_k}(T-\tau))}{\sqrt{\lambda_k}})(\frac{u}{\tau})^{H-\frac12}(u-\tau)^{H-\frac32}du\Big]^2d\tau\\
&& +\int_{t_*}^T\Big[\int_{\tau+t_*}^T(\frac{\sin(\sqrt{\lambda_k}(T-u))}{\sqrt{\lambda_k}}-\frac{\sin(\sqrt{\lambda_k}(T-\tau))}{\sqrt{\lambda_k}})(\frac{u}{\tau})^{H-\frac12}(u-\tau)^{H-\frac32}du\Big]^2d\tau\\
&=:&K_1+K_2+K_3+K_4.
\end{eqnarray*}
For $K_1$ and $K_3$, we have from the differential mean value theorem that 
\begin{eqnarray*}
K_1&\lesssim&\int_0^{t_*}\Big[\int_\tau^{t_*}(\frac{u}{\tau})^{H-\frac12}(u-\tau)^{H-\frac12}du\Big]^2d\tau\\
&\lesssim&\int_0^{t_*}\Big[\int_\tau^{t_*}(u-\tau)^{H-\frac12}du\big]^2d\tau\lesssim\int_0^{t_*}(t_*-\tau)^{2H+1}d\tau\lesssim t_*^{2H+2}
\end{eqnarray*}
and
\begin{eqnarray*}
K_3 &\lesssim&\int_{t_*}^T\Big[\int_\tau^{\tau+t_*}(u-\tau)^{H-\frac12}du\Big]^2d\tau\\
&\lesssim&\int_{t_*}^Tt_*^{2H+1}d\tau\lesssim t_*^{2H+1}.
\end{eqnarray*}
For $K_2$ and $K_4$, it follows from straightforward calculations that
\begin{eqnarray*}
K_2&\lesssim&\frac{1}{\lambda_k}\int_0^{t_*}\Big[\int_{t_*}^T(u-\tau)^{H-\frac32}du\Big]^2d\tau\\
&\lesssim&\frac{1}{\lambda_k}\int_0^{t_*}(t_*-\tau)^{2H-1}d\tau\lesssim\frac{1}{\lambda_k}t_*^{2H}
\end{eqnarray*}
and
\begin{eqnarray*}
K_4&\lesssim&\frac{1}{\lambda_k}\int_{t_*}^T\left[\int_{\tau+t_*}^T(u-\tau)^{H-\frac32}du\right]^2d\tau\\
&\lesssim&\frac{1}{\lambda_k}\int_{t_*}^T((T-\tau)^{2H-1}+t_*^{2H-1})d\tau\lesssim\frac{1}{\lambda_k}+\frac{1}{\lambda_k}t_*^{2H-1}.
\end{eqnarray*}
Combining the above estimates and choosing $t_*=\lambda_k^{-1}$, we deduce
\begin{align*}
I_3(T)\lesssim t_*^{2H+2}+\frac{1}{\lambda_k}t_*^{2H}+t_*^{2H+1}+\frac{1}{\lambda_k}+\frac{1}{\lambda_k}t_*^{2H-1}\lesssim\lambda_k^{-2H}.
\end{align*}

For the case $H\in(\frac{1}{2},1)$, we have
\begin{eqnarray*}
&&\mathbb{E}\Big[
\Big(\int_0^T\frac{\sin(\sqrt{\lambda_k}(T-\tau))}{\sqrt{\lambda_k}}dB^H(\tau)\Big)^2\Big]\\
&=&\alpha_H\int_0^T\int_0^T\frac{\sin(\sqrt{\lambda_k}(T-p))}{\sqrt{\lambda_k}}\frac{\sin(\sqrt{\lambda_k}(T-q))}{\sqrt{\lambda_k}}|p-q|^{2H-2}dpdq\\
&\lesssim&\frac{1}{\lambda_k}\int_0^T\int_0^T|p-q|^{2H-2}dpdq
\lesssim\frac{1}{\lambda_k}T^{2H}.
\end{eqnarray*}

Hence, we conclude for any $H\in (0,1)$ that
\begin{align*}
\mathbb{E}\Big[
\Big(\int_0^T\frac{\sin(\sqrt{\lambda_k}(t-\tau))}{\sqrt{\lambda_k}}dB^H(\tau)\Big)^2\Big]\lesssim\lambda_k^{-\gamma},
\end{align*}
where $\gamma=\min\{2H,1\}$. When $k\rightarrow\infty$, $\lambda_k^{-\gamma}\rightarrow0$, which shows that it is unstable to recover $|g|$.

\begin{theorem}\label{theroem3f}
The inverse problem is unstable to recover the source terms $f$ and $\pm g$. Moreover, the following estimates hold:
\begin{equation*}%\label{eq:ille1f}
\left|\int_0^Th(\tau)\frac{\sin(\sqrt{\lambda_k}(T-\tau))}{\sqrt{\lambda_k}}d\tau\right|\lesssim\frac{1}{\sqrt{\lambda_k}}
\end{equation*}
and
\begin{equation*}%\label{eq:ille2f}
\mathbb{E}\Big[
\Big(\int_0^T\frac{\sin(\sqrt{\lambda_k}(T-\tau))}{\sqrt{\lambda_k}}dB^H(\tau)\Big)^2\Big]\lesssim\lambda_k^{-\gamma},
\end{equation*}
where $\gamma=\min\{2H,1\}$.
\end{theorem}

\section{Numerical experiments}

In this section, we present some numerical experiments for the one-dimensional problem where $D=[0,\pi]$. For some fixed integers $N_t$ and $N_x$, we define the time and space step-sizes $h_t=T/N_t$, $ h_x=\pi/N_x$ and nodes
\[
 t_n=nh_t,~n\in\{0,1,2,\ldots,N_t\},\quad x_i=ih_x,~i\in\{0,1,2,\ldots,N_x\}.
 \]

For the direct problem, the second order central difference is utilized to generate the synthetic data. Let $u^n_i$ be the numerical approximation to $u(x_i, t_n)$. Then we obtain the following explicit scheme:
\begin{align}\label{ed}
\frac{u_i^{n-1}-2u_i^n+u_i^{n+1}}{h_t^2}-\frac{u_{i-1}^n-2u_i^n+u_{i+1}^n}{h_x^2}=f(x_i)h(t_n)+g(x_i)\frac{B^H(t_{n})-B^H(t_{n-1})}{t_{n}-t_{n-1}}.
\end{align}
We use the ghost point method to handle the discretization of the initial condition $u_t(x,0)$. The virtual point $u_i^{-1}$ is introduced and the first order central difference is adopted, i.e., $\partial_t u|_{t=0}=(u_i^1-u_i^{-1})/(2h_t)$. From (\ref{ed}), when $n=0$, another equation about $u_i^{-1}$ can be obtained. Then we may use the two equations to get rid of $u_i^{-1}$ and obtain $u_i^1$.

For the inverse problem, the coefficients $f_k$ and $g_kg_l$ can be recovered by using (\ref{Iexp}) and (\ref{Icvar}), respectively. Once the coefficients are available, the source functions $f$ and $g^2$ can be expressed as
\[
f=\sum_{k=1}^\infty f_k\varphi_k,\quad g^2=\sum_{k,l=1}^\infty g_kg_l\varphi_k\varphi_l.
\]
Noting that computing $|g|$ and $g^2$ are equivalent, here we consider computing $g^2$. Since the inverse problem is ill-posed, we truncate above series by keeping the first $N$ terms as a regularization.

In the numerical experiments, we choose $N=9$, $N_t=2^{13}, N_x=100$, and $T=1$. The exact functions in \eqref{directp} are chosen as
\begin{equation*}
h(t)=1, \quad f(x)=\sin(3x), \quad g(x)=\exp(-(x-0.5\pi)^2).
\end{equation*}
We compute 1000 sample paths when simulating the covariance of the solution. In addition, the data is polluted by a uniformly distributed noise with level $\delta$.

\begin{figure}
  \centering
  \includegraphics[width=0.4\textwidth]{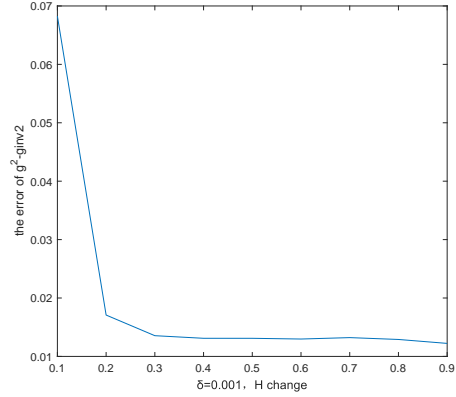}\\
  \caption{The relative errors of reconstruction for $g^2$ with different $H$ and a fixed $\delta=0.001$.}\label{fig2}
\end{figure}

\begin{table}\label{tabl1}
\centering
  \caption{The relative errors of reconstructions for $f$ and $g^2$ with different $\delta$ and a fixed $H=0.9$.}%\label{}
\begin{tabular}{lllllll}
  \hline
  \hline
  $\delta$ & $0.001$ & $0.005$ & $0.01$ & $0.05$ & $0.1$ \\
  \hline
  $f$ & $0.0260$ & $0.0261$ & $0.0261$ & $0.0286$ & $0.0837$ \\
  \hline
  $g^2$& $0.0141$ & $0.0147$ & $0.0237$ & $0.0729$ & $0.0683$ \\
  \hline
  \hline
\end{tabular}
\end{table}

We report the numerical results for different sets of parameters $(H, \delta)$. Figure \ref{fig2} shows the results for the relative errors of $g^2$ with different $H$ and a fixed $\delta=0.001$.  For $H=0.9$, the results for the relative errors of $f$ and $g^2$ with different $\delta$ are given in Table 1. Figure \ref{fig3} plots the results of
$\{H=0.9, \delta=0.001\}, \{H=0.2, \delta=0.001\}$, and $\{H=0.9, \delta=0.01\}$. For $H=0.5$ and $\delta=0.1, 0.01, 0.001$, the exact and reconstructed solutions are given in Figure \ref{fig4}. Based on the numerical experiments, it can be observed that the reconstructions would be more accurate if the problem is more regular, i.e., $H$ is larger; if the noise level $\delta$ is smaller, the results would also be better, which exactly implies the ill-posedness of the inverse problem.

\begin{figure}
  \centering
  \includegraphics[width=0.4\textwidth]{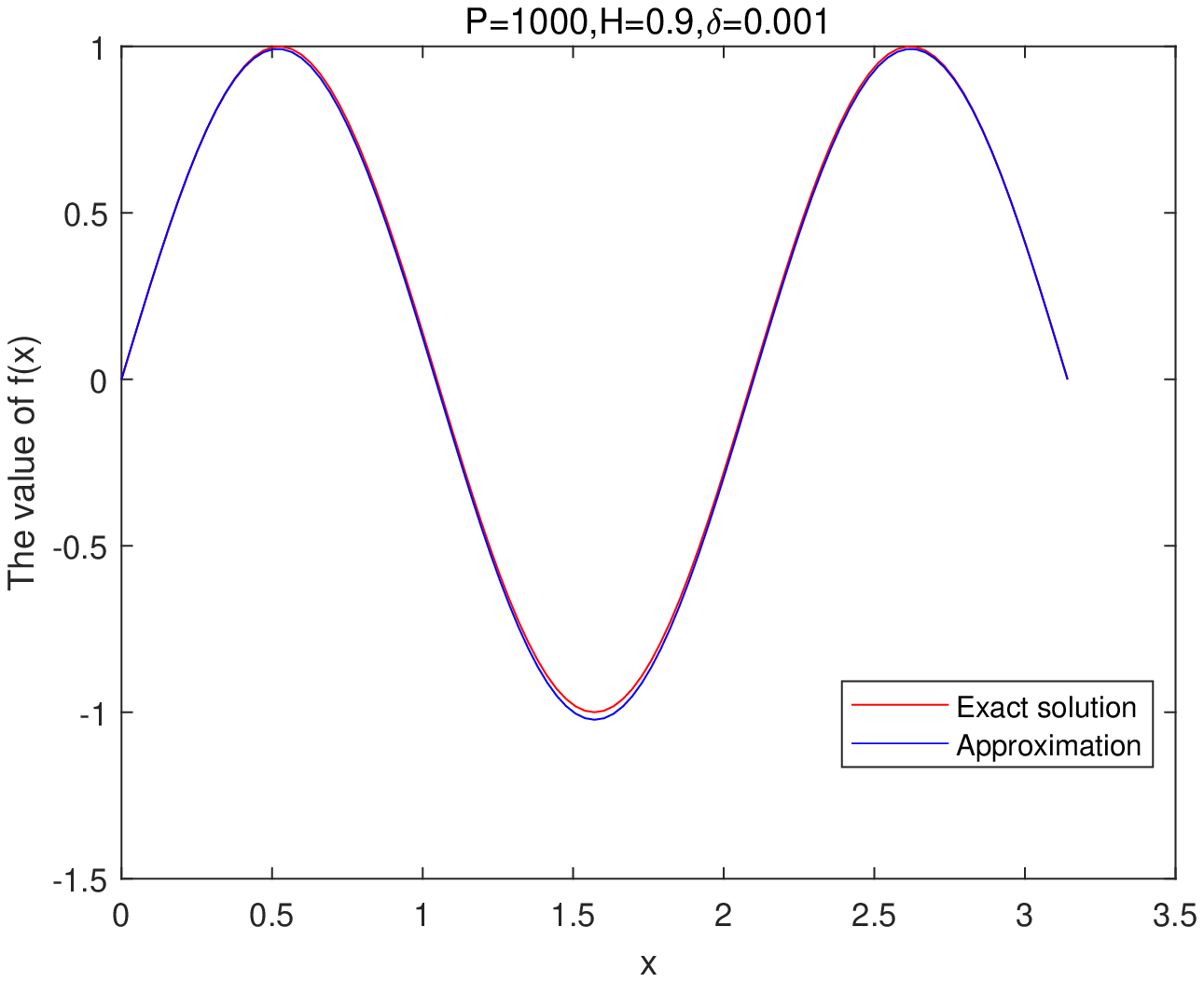}
  \includegraphics[width=0.4\textwidth]{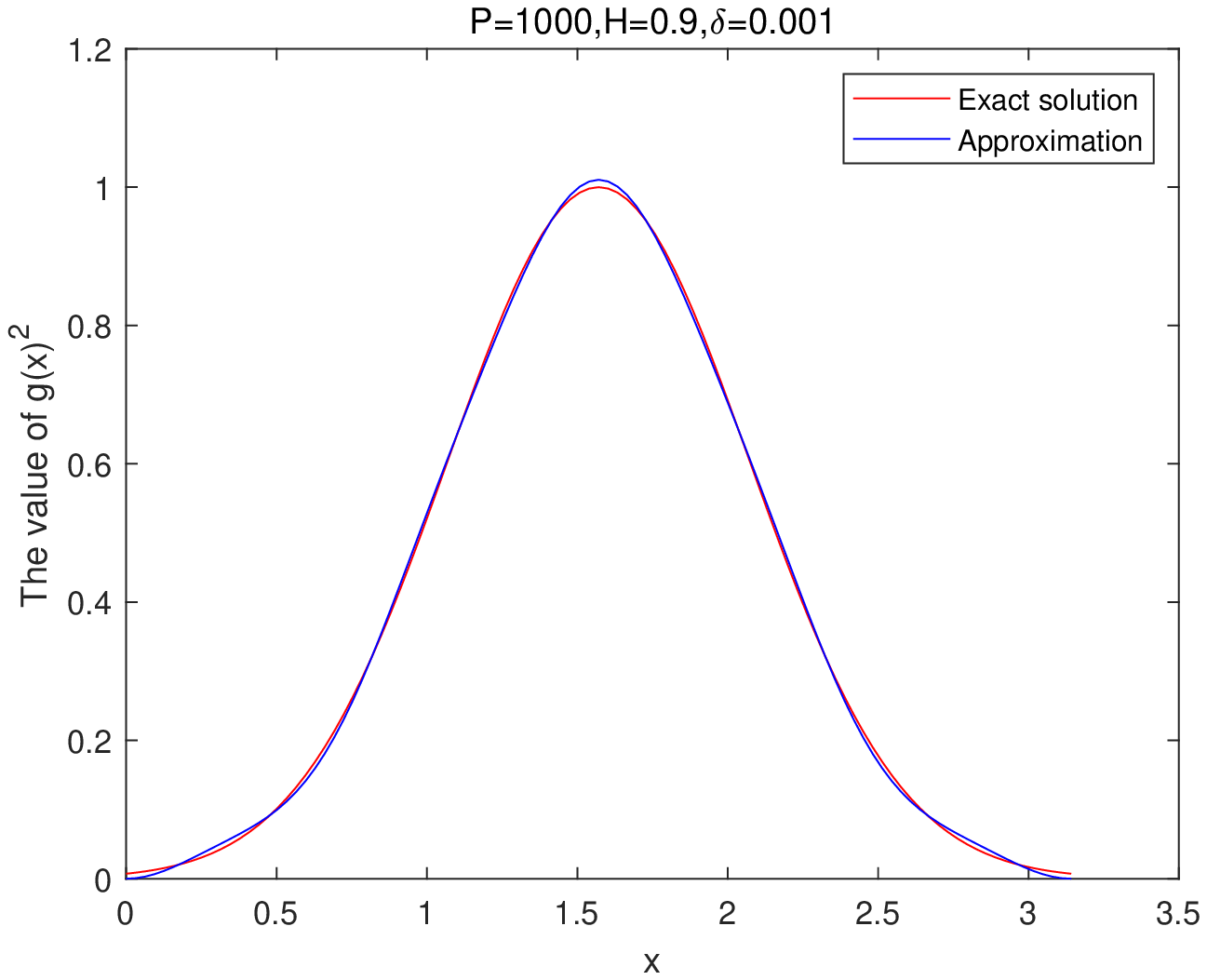}\\
  \includegraphics[width=0.4\textwidth]{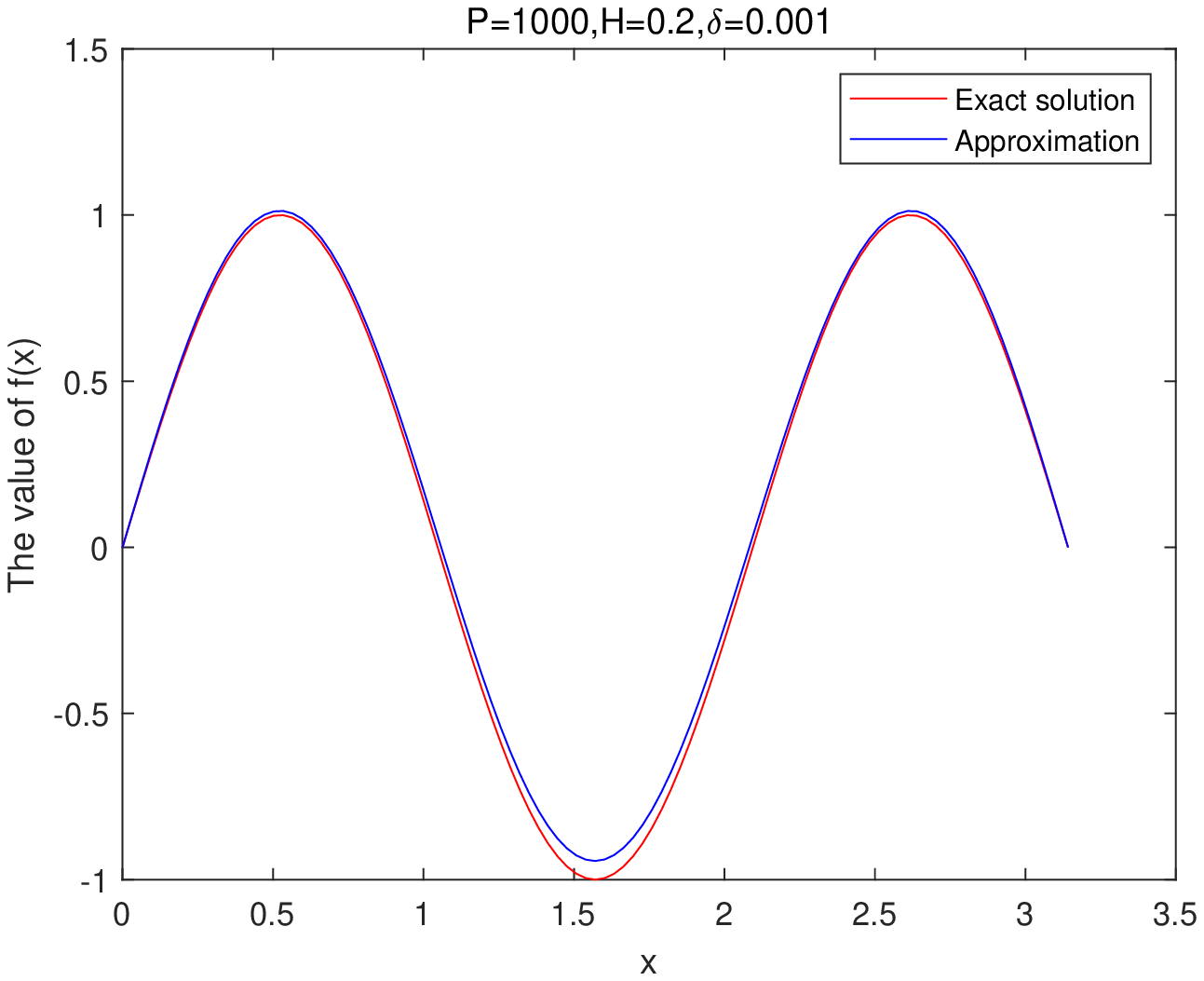}
  \includegraphics[width=0.4\textwidth]{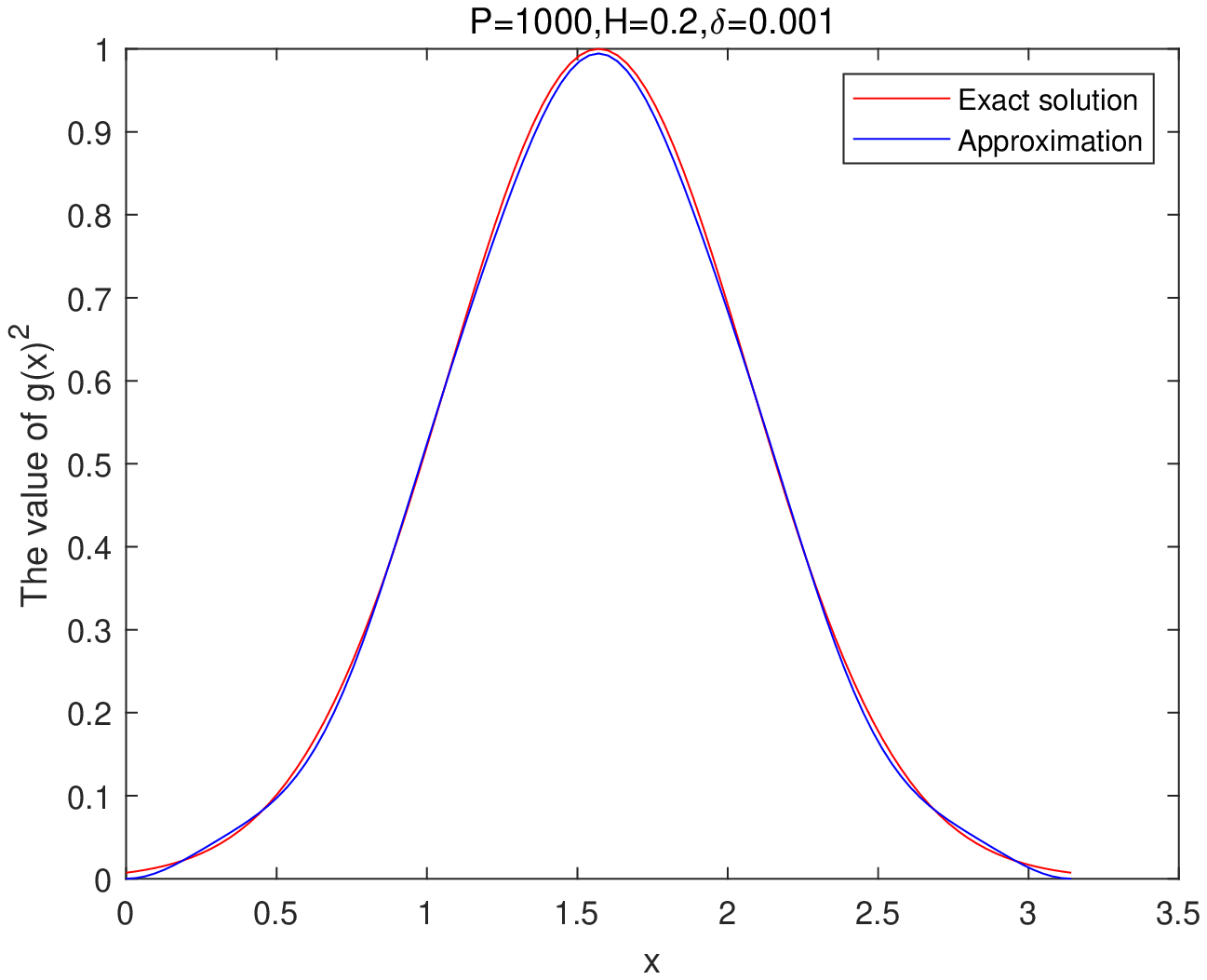}\\
  \includegraphics[width=0.4\textwidth]{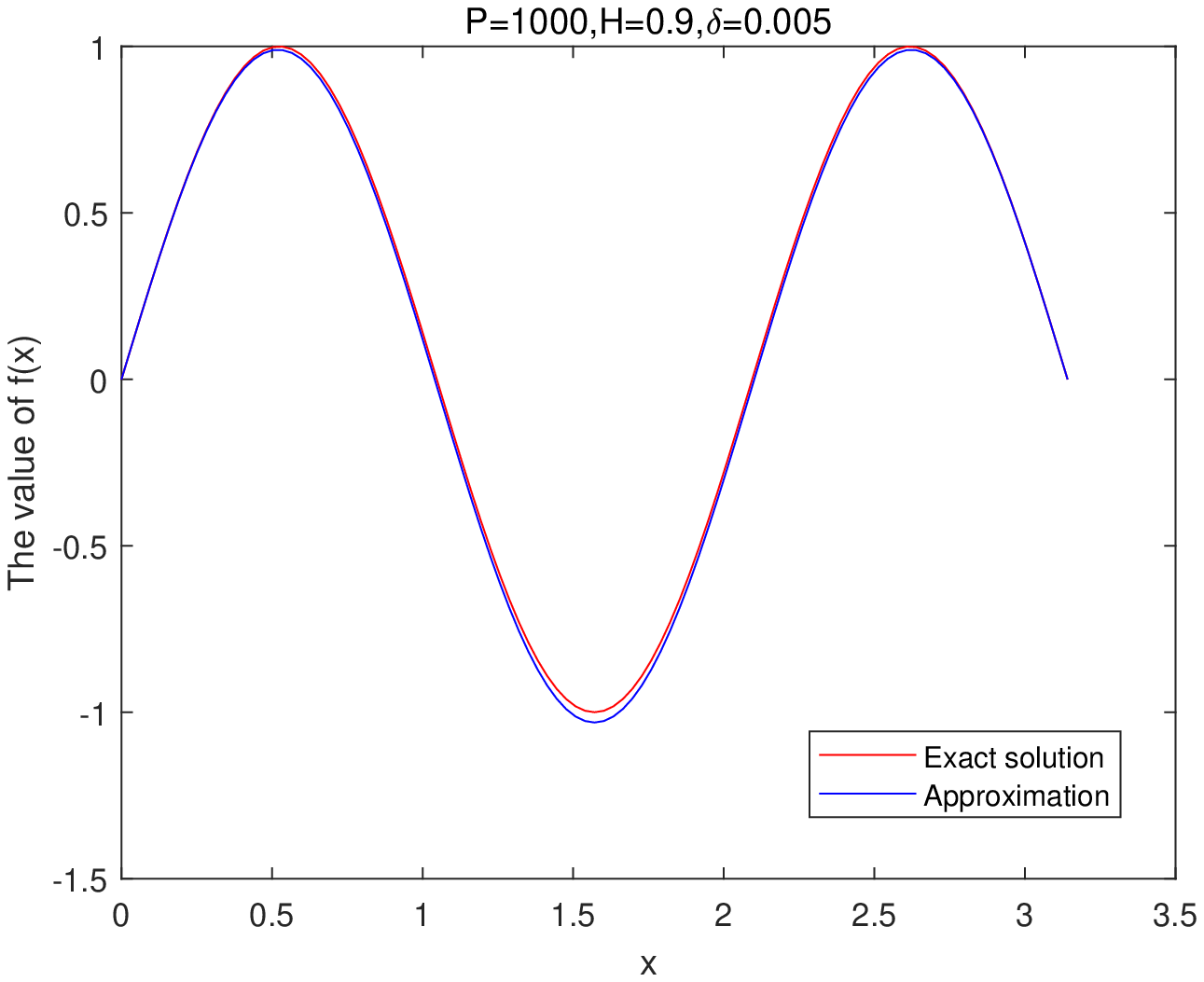}
  \includegraphics[width=0.4\textwidth]{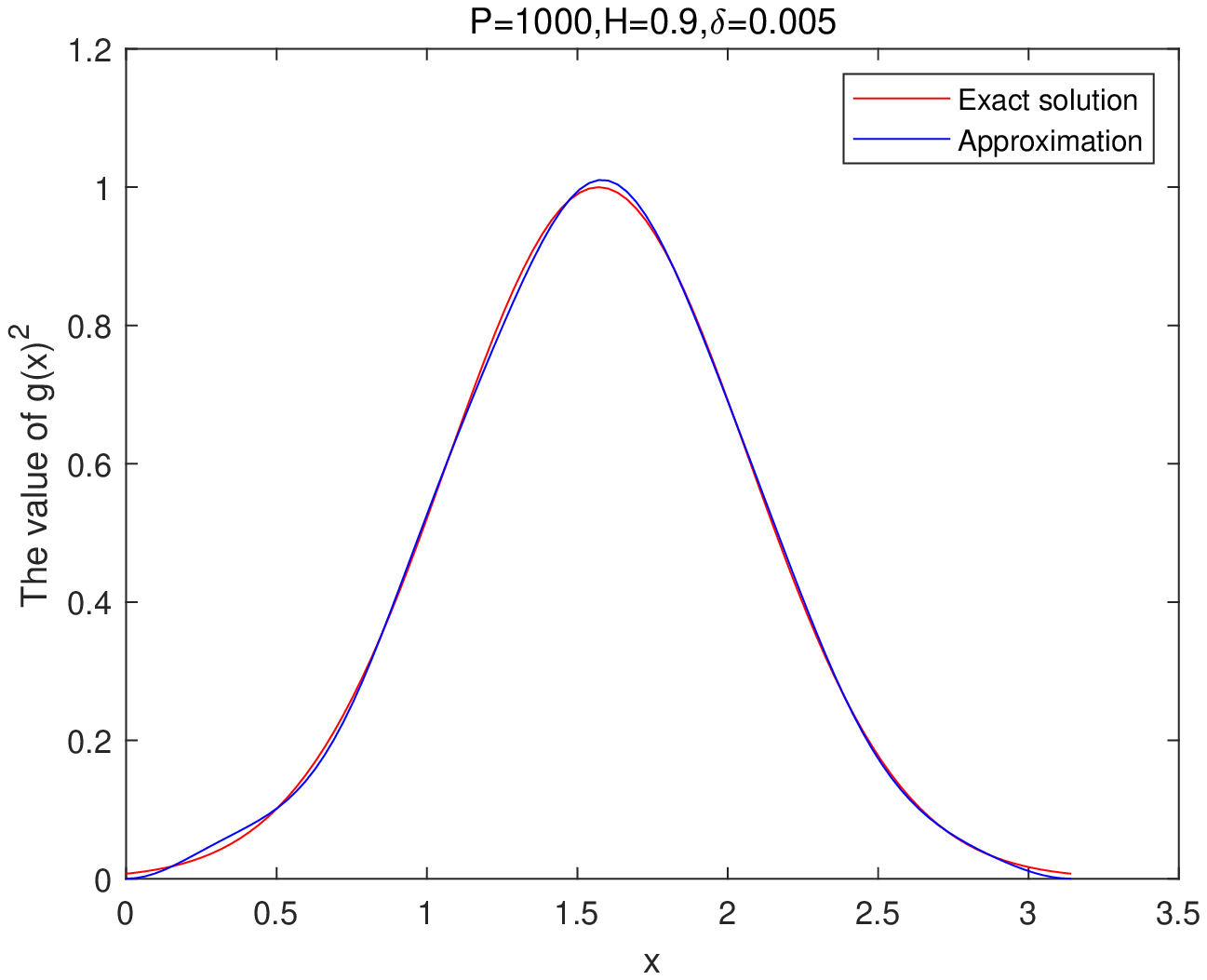}\\
   \caption{The exact solution is plotted against the reconstructed solutions with $\{H=0.9,\delta=0.001\}, \{H=0.2,\delta=0.001\}$ and $\{H=0.9,\delta=0.01\}$. (left) $f(x)$; (right) $g^2(x)$.}\label{fig3}
\end{figure}

\begin{figure}
  \centering
  \includegraphics[width=0.4\textwidth]{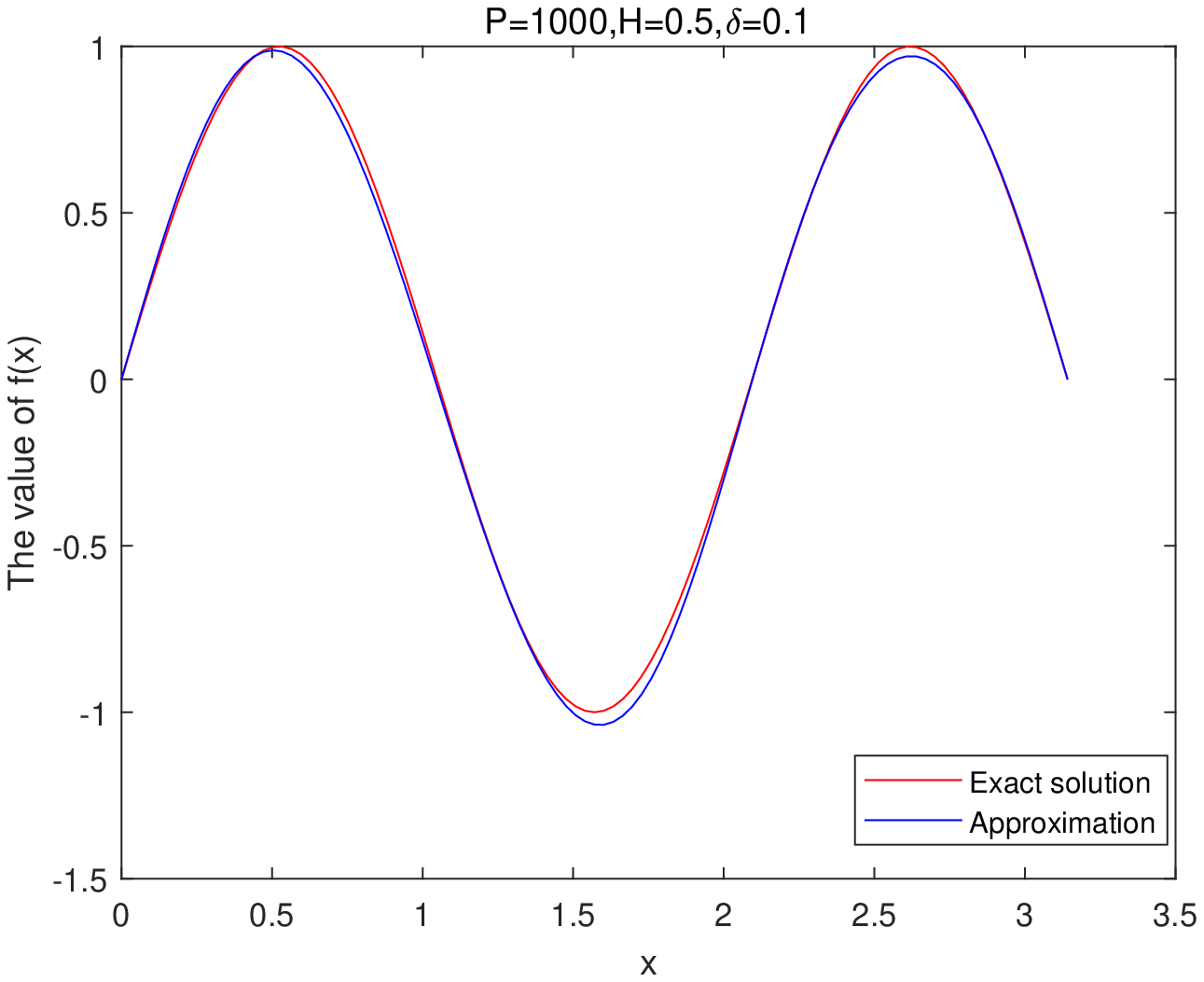}
  \includegraphics[width=0.4\textwidth]{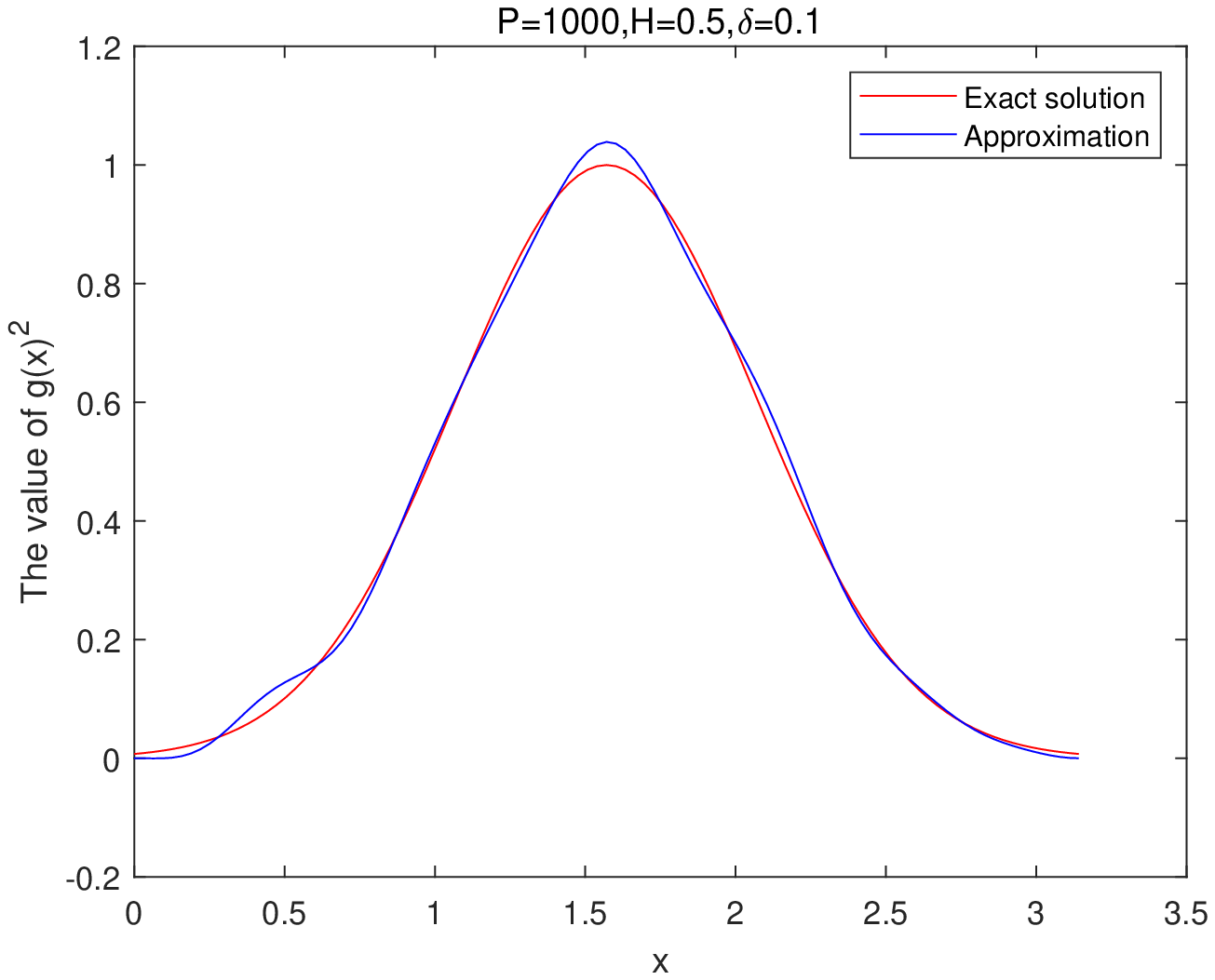}\\
  \includegraphics[width=0.4\textwidth]{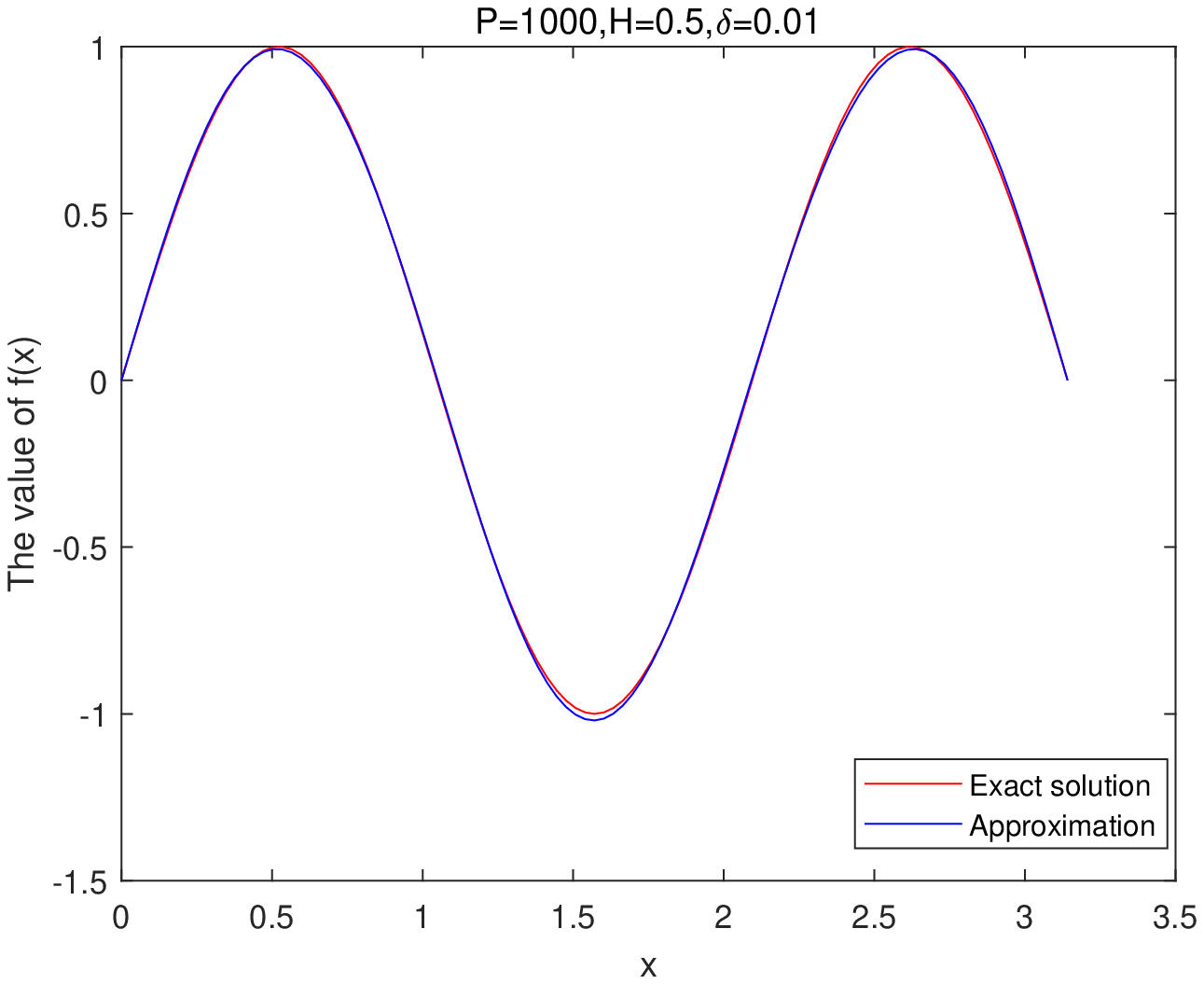}
  \includegraphics[width=0.4\textwidth]{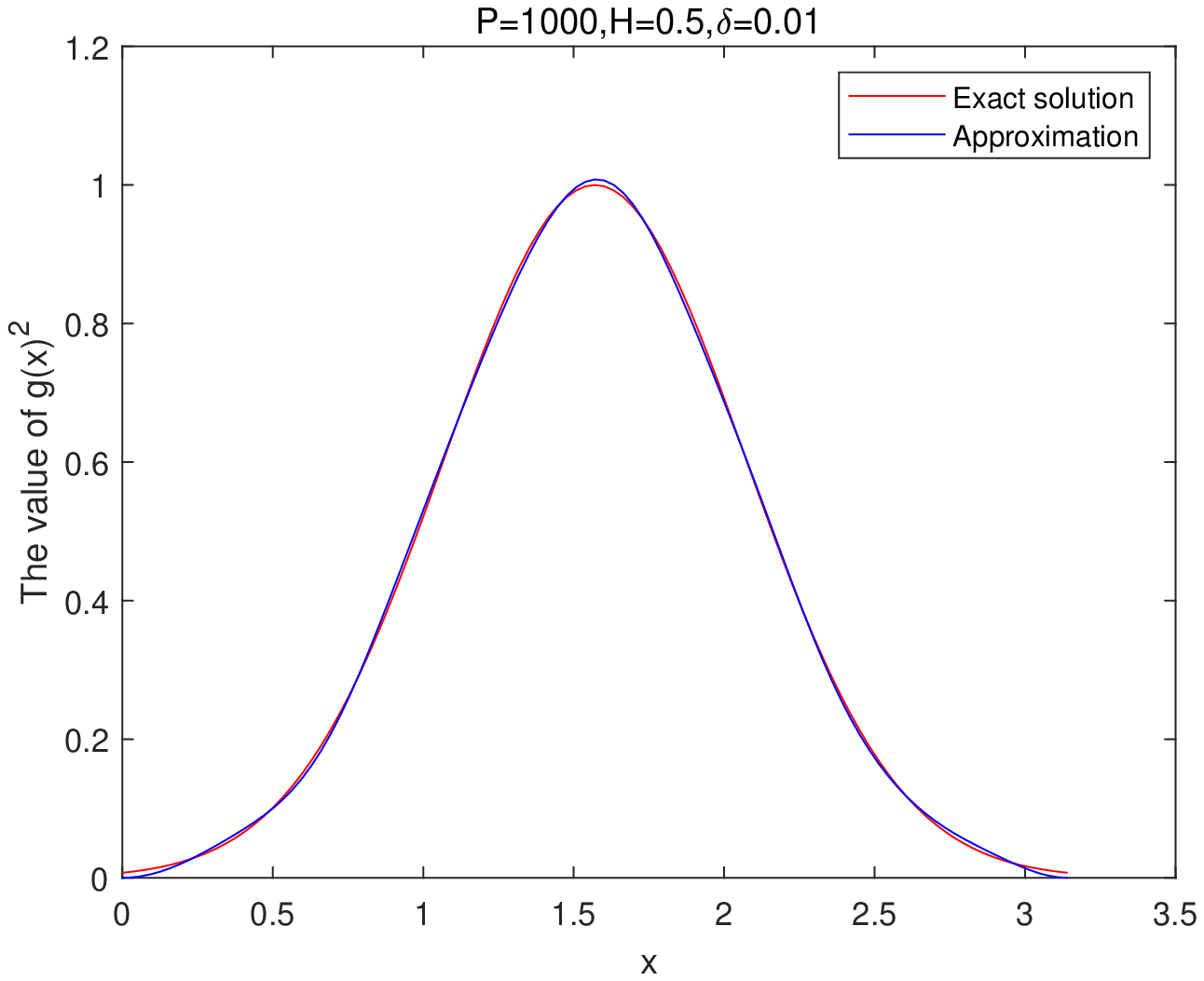}\\
  \includegraphics[width=0.4\textwidth]{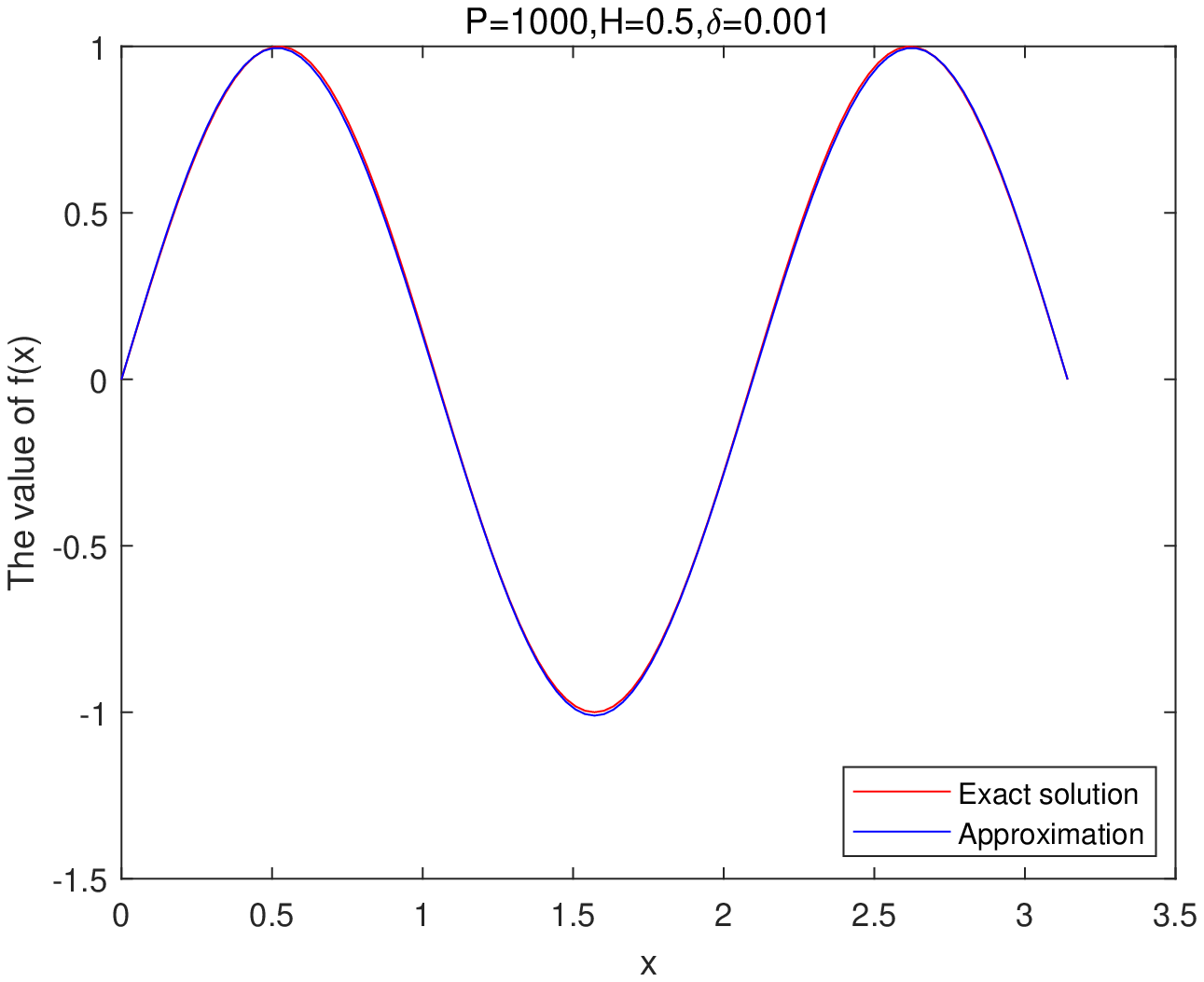}
  \includegraphics[width=0.4\textwidth]{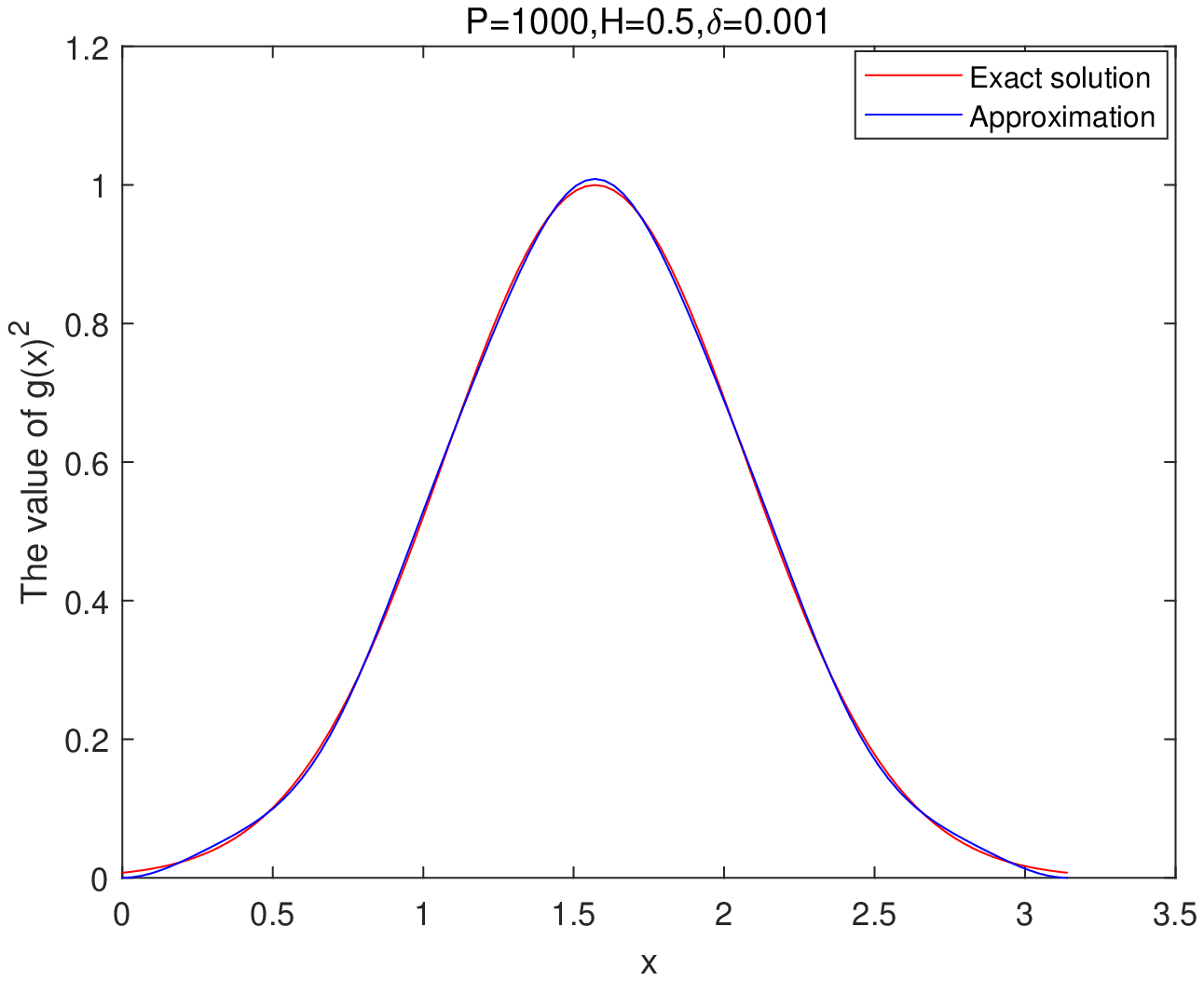}\\
   \caption{The exact solution is plotted against the reconstructed solutions with $H=0.5$ and $\delta=0.1, 0.01, 0.001$. (left) $f(x)$; (right) $g^2(x)$. }\label{fig4}
\end{figure}

\section{Conclusion}

In this paper, we have studied an inverse random source problem for the wave equation driven by the fBm.  We show that the direct problem is well-posed and the inverse problem is ill-posed in the sense that a small deviation of the data may lead to a huge error in the reconstruction. Moreover, for the one-dimensional case, the inverse problem is shown to has a unique solution for the white noise case, i.e., $H=\frac12$. It is unclear if the uniqueness can still hold for the general Hurst index due to the highly oscillatory kernel function for large $k$. We will investigate the uniqueness issue for the ISP of the higher dimensional stochastic wave equation driven by the fBm in the future.

\end{document}